\documentclass[12pt,a4paper,reqno]{amsart}

\usepackage[active]{srcltx}
\usepackage{a4wide}
\usepackage{amssymb, amsmath, mathtools}
\usepackage{graphicx}
\usepackage{float}
\usepackage[active]{srcltx}
\usepackage[ansinew]{inputenc}

\usepackage{hyperref}

\theoremstyle{plain}% default
\newtheorem{theorem}{Theorem}[section]
\newtheorem{maintheorem}{Theorem}
\newtheorem{lemma}[theorem]{Lemma}
\newtheorem{proposition}[theorem]{Proposition}

\newtheorem{maincorollary}[maintheorem]{Corollary}

\theoremstyle{definition}

\newtheorem{example}[theorem]{Example}

\theoremstyle{remark}
\newtheorem{remark}[theorem]{Remark}%[section]

\def\R{\ensuremath{\mathbb R}}

\def\top{\operatorname{top}}

\numberwithin{equation}{section}

\begin{document}

\title[Invariant probability measures for impulsive semiflows]{Invariant probability measures and non-wandering sets for impulsive semiflows}
\author[J. F. Alves]{Jos\'{e} F. Alves}
\address{Jos\'{e} F. Alves\\ Centro de Matem\'{a}tica da Universidade do Porto\\ Rua do Campo Alegre 687\\ 4169-007 Porto\\ Portugal}
\email{jfalves@fc.up.pt} \urladdr{http://www.fc.up.pt/cmup/jfalves}

\author[M. Carvalho]{Maria Carvalho}
\address{Maria Carvalho\\ Centro de Matem\'{a}tica da Universidade do Porto\\ Rua do
Campo Alegre 687\\ 4169-007 Porto\\ Portugal}
\email{mpcarval@fc.up.pt}

\date{\today}
\thanks{The authors were partially funded by the
European Regional Development Fund through the program COMPETE and by the Portuguese Government through the FCT - Funda{\c c}{\~a}o para a Ci{\^e}ncia e a Tecnologia under the project PEst-C/MAT/UI0144/2013.
JFA was also partially supported by Funda\c c\~ao Calouste Gulbenkian and the project PTDC/MAT/120346/2010.}
\keywords{Impulsive Dynamical System, Invariant Measure, Non-wandering Set}
\subjclass[2010]{37A05}

\begin{abstract}
We consider  impulsive dynamical systems defined on compact metric spaces and their respective impulsive semiflows. We establish sufficient conditions for the existence of probability measures which are invariant by such impulsive semiflows. We also deduce the forward invariance of their non-wandering sets except the discontinuity points.
\end{abstract}

\maketitle

\setcounter{tocdepth}{2}

\tableofcontents %\clearpage

\section{Introduction}

An impulsive semiflow is prescribed by three ingredients: a continuous semiflow on a space $X$ which governs the state of the system between impulses; a set $D \subset X$ where the flow undergoes some abrupt perturbations, whose duration is, however, negligible in comparison with the time length of the whole process; and an impulsive function $I:D\to X$ which specifies how a jump event happens each time a trajectory of the flow hits $D$, and whose action may be a source of discontinuities on the trajectories.

Dynamical systems with impulse effects seem to be the most adequate mathematical models to describe real world phenomena that exhibit sudden changes in their states. For example, the theoretical characterizations of wormholes \cite{T94}, also called Einstein-Rosen bridges, seem to fit the description of the traverse effects an impulsive function $I$ acting on a set $D$ induces on a semiflow, thereby possibly creating odd shortcuts in space-time \cite{V95}. While at present it appears unlikely that nature allows us to observe a wormhole, these hypothetical entities, with unusual and inherently unstable topological, geometrical and physical properties, show up as valid solutions of the Einstein field equations for the gravity. We also refer the reader to the reference \cite{LBS89}, where other examples of nature evolution processes are analyzed within the new branch of differential equations with impulses; in addition, see
\cite{BS89,CGR06,d05,DB90,EFL91,JL07,L99,N99,R97,Z94}.

One of the major developments so far on the theory of impulsive dynamical systems has been to extend the classical theorem on existence and uniqueness of solutions and to establish sufficient conditions to ensure a complete characterization and some asymptotic stability of the limit sets \cite{B07,BF07,BF08,C04b,K94a}. Meanwhile, a significant progress in the study of dynamical systems has been achieved due to a remarkable sample of the so-called ergodic theorems \cite{W82} which concern the connection between the time and the spatial averages of observable measurable maps along orbits, and whose fundamental request is the existence of an invariant probability measure. Given the impact on applications of these ergodic results, we addressed the question of whether an impulsive flow acting on a compact metric space preserves one such a measure. To our best knowledge, no research on impulsive dynamical systems has been dealing with this timely issue.

Aware of the fact that discontinuous discrete dynamical systems may preserve no invariant probability measure and bearing in mind that, at large, an impulsive flow exhibits discontinuities, we looked for a quotient structure where the motion of the impulsive semiflow could be redesigned as a continuous flow, although defined on a different compact metric space. Along this procedure, which was inspired by Example~\ref{ex.with}, we were faced with an overriding problem. After getting an invariant probability measure on the quotient space, we had to lift it to the original discontinuous flow, with respect to which recurrence to positive measure subsets should occur \cite{H56} and the non-wandering set had to receive full measure. Yet, without continuity of the flow, the invariance of the sets of recurrent or non-wandering points is not guaranteed. To overcome these difficulties we turned to properties of the function $I$. Hopefully, we expected to find reasonable, not too demanding, assumptions on the impulse outcome and its time schedule that allowed us to find an invariant probability measure for an impulsive flow. This is the content of our main result.

We also obtain some interesting properties for non-wandering sets of impulsive semiflows. The relevance of non-wandering sets in dynamical systems goes far beyond their intrinsic dynamical value, for they actually play a nontrivial role in uniformly hyperbolic dynamics and the structural stability of systems \cite{P88,S70}. From a measure theoretical viewpoint, within an ergodic approach to dynamical systems, non-wandering sets are also the natural place to look for invariant probability measures, since any invariant probability measure by a semiflow must be supported on the non-wandering set.

We believe this work is a promising starting point that will encourage further study of many relevant ergodic properties of impulsive semiflows, as a formula for the topological entropy, a variational principle, the existence of physical measures, among others.

\subsection{Non-wandering sets} We start defining precisely two of the main objects that we shall use in this work, namely the notions of semiflow and non-wandering set.
Given a metric space $X$, we say that  $\phi: \R^+_0 \times X\to X$ is a \emph{semiflow} if for all $x\in X$ and all $s,t\in\R^+_0$ we have
\begin{enumerate}
\item $\phi_0(x)= x$,
\item $\phi_{t+s}(x)=\phi_t(\phi_s(x))$,
\end{enumerate}
where $\phi_t(x)$ stands for  $\phi(t,x)$. The curve defined for $t\ge0$ by $\phi_t(x)$ is called the \emph{$\phi$-trajectory} of the point $x\in X$.
We say that a set $A\subset X$ is  \emph{forward invariant} under   $\phi$  if we have $\phi_t(A) \subseteq A$ for all $t \geq 0$.

A point $x\in X$ is said to be \emph{non-wandering} for a semiflow~$\phi$ if, for every neighborhood $U$ of $x$ and any $T>0$, there exists $t\geq T$ such that $\phi_t^{-1}(U)\cap U\neq\emptyset$. The \emph{non-wandering set} of  $\phi$ is defined as
  $$\Omega_\phi= \{x\in X: \text{ $x$ is non-wandering for $\phi$\,}\}.$$
It follows immediately from this definition that the non-wandering set $\Omega_\phi$  of a semiflow  is closed. Moreover, $\Omega_\phi$   contains the set of limit points of the semiflow, which is clearly nonempty when $X$ is compact. Therefore, non-wandering sets of semiflows defined on compact sets are always nonempty and compact.

\subsection{Impulsive dynamical systems} Given a continuous semiflow $\varphi:\R^+_0\times X\to X$, a compact set $D\subset X$  and a continuous function $I:D\to X$ we define a new function $\tau_1:X\to [0,+\infty]$ as
$$
\tau_1(x)=
\begin{dcases}
\inf\left\{t> 0:\varphi_t(x)\in D\right\} ,& \text{if } \varphi_t(x)\in D\text{ for some }t>0;\\
+\infty, & \text{otherwise.}
\end{dcases}
$$
This corresponds to assigning to each point the time its $\varphi$-trajectory needs to spend to hit the set $D$.
Assuming $\tau_1(x)>0$ for all $x\in X$, we define the \emph{impulsive trajectory} $\gamma_x:[0,T(x)[\,\to X$  and the subsequent \emph{impulsive times} of $x\in X$ according to the following rules:
\begin{enumerate}
\item If $0\le t<\tau_{1}(x)$, then we set  $\gamma_x(t)=\varphi_t(x)$.
\item If $\tau_1(x)<\infty$, then we proceed inductively:
\begin{enumerate}
\item Firstly we set
  $$\gamma_x(\tau_1(x))=I(\varphi_{\tau_1(x)}(x)).$$
  Defining the second impulsive time of $x$ as
  $$\tau_2(x)=\tau_1(x)+\tau_1(\gamma_x(\tau_1(x))),$$
 we set
         $$\gamma_x(t)=\varphi_{t-\tau_1(x)}(\gamma_x(\tau_1(x))),\quad\text{for }\tau_1(x)<t<\tau_2(x).$$
         \item Assuming that $\gamma_x(t)$ is defined for $t<\tau_{n}(x)$, for some $n\ge 2$, we set
  $$\gamma_x(\tau_{n}(x))=I(\varphi_{\tau_n(x)-\tau_{n-1}(x)}(\gamma_x({\tau_{n-1}(x)}))).$$
  Defining the $(n+1)^{\text{th}}$ impulsive time of $x$ as
  $$\tau_{n+1}(x)=\tau_n(x)+\tau_1(\gamma_x(\tau_n(x))),$$
  we set
         $$\gamma_x(t)=\varphi_{t-\tau_i(x)}(\gamma_x(\tau_n(x))),\quad\text{for }\tau_n(x)<t<\tau_{i+1}(x).$$
\end{enumerate}
Finally, we define the time length of the trajectory of $x$ as
 $$T(x)=\sup_{n\ge 1}\,\{\tau_n(x)\}.
 $$
\end{enumerate}
In general, given an initial condition $x\in X$, both situations $T(x)=+\infty$ or $T(x)<+\infty$ are possible; see e.g. \cite[Example 2.6]{C04a} and Remark~\ref{re.infty} below.
In this work  we are going to consider only the first possibility, meaning that  the impulsive trajectories of points in $X$ are defined for all $t\ge0$.

\begin{remark}\label{re.infty} Under the fairly reasonable condition $I(D)\cap (D)=\emptyset$, for instance, we have  $T(x)=\infty$ for all $x\in X$. Indeed, it has been proved in \cite[Theorem~2.7]{C04a} that $\tau_1$ is always lower semicontinuous on the set $X\setminus D$. As $D$ is compact and $I$ is continuous, then $I(D)$ is compact. Supposing that $I(D) \cap D=\emptyset$, then by the lower semicontinuity of $\tau_1$ on $X\setminus D$  there must be some $\alpha>0$ such that $\tau_1(x)>\alpha$ for all $x\in I(D).$ This clearly implies that $T(x)=\infty$ for all $x\in X$.
\end{remark}

We say that $(X,\varphi,D,I)$ is an \emph{impulsive dynamical system} if
 $$\tau_1(x)>0\quad\text{and}\quad T(x)=+\infty, \quad\text{for all $x\in X$}.$$
 We call $D$ the \emph{impulsive set}, $I$ the \emph{impulsive function} and $\tau_1$ the \emph{first impulsive time} of the impulsive dynamical system.
The  \emph{(impulsive) semiflow} $\phi$ of an impulsive dynamical system $(X,\varphi, D, I)$ is defined as
$$
\begin{array}{cccc}
        \phi:  &  \mathbb{R}^+_0 \times X & \longrightarrow &X \\
        & (t,x) & \longmapsto & \gamma_x(t), \\
        \end{array}$$
where $\gamma_x(t)$ is the impulsive trajectory of $x$ determined by $(X,\varphi,D, I)$.
It has been  proved in  \cite[Proposition 2.1]{B07} that $\phi$ is indeed a semiflow, though not necessarily continuous.

\subsection{Invariant probability measures}\label{se.im} A map between two topological spaces is called \emph{measurable} if the pre-image of any Borel set is a Borel set. We say that an invertible  map is  \emph{bimeasurable} if both the map and its inverse are measurable.
 Notice that the measurability of a semiflow $\phi:\R^+\times X\to X$ gives in particular that $\phi_t$ is measurable for each $t\ge0$.
%   $\phi$ being measurable we mean that $\phi^{-1}(A)$ is a Borel set in $\R_0^+\times X$ for every Borel set $A\subset X$.

A probability measure $\mu$ on the Borel sets of a topological space $X$ is said to be   \emph{invariant} by a semiflow~$\phi$ (or \emph{$\phi$-invariant}) if $\phi$ is measurable and
   $$\mu(\phi_t^{-1}(A))=\mu(A),$$
for every Borel set $A\subset X$ and every $t\geq 0$. We denote by $\mathcal M(X)$ the set of all probability measures on the Borel sets of $X$ and by $\mathcal M_\phi(X)$ the set of those measures in $\mathcal M(X)$ which are $\phi$-invariant.

Given a measurable map $f:X\to Y$ between two topological spaces $X$ and $Y$ we introduce the \emph{push-forward} map
$$
\begin{array}{ccc}
        f_*:   \mathcal M(X) & \longrightarrow & \mathcal M(Y)\\
        \mu & \longmapsto & f_*\mu \\
        \end{array}$$
with $f_*\mu$ defined for any $\mu\in  \mathcal M(X)$ and any Borel set $B\subset Y$ as
 $$f_*\mu(B)=\mu(f^{-1}(B)).$$
 Observe that  $\mu\in \mathcal M(X)$  belongs to  $ \mathcal M_\phi(X)$ if and only if $(\phi_t)_*\mu=\mu$ for all $t\ge0$.

The \emph{support} of a measure  $\mu\in \mathcal M(X)$  is defined as the set of points $x\in X$ such that $\mu(U)>0$ for any neighborhood $U$ of $x$. If a probability measure is invariant by a semiflow, then its support must necessarily be contained in the non-wandering set of that semiflow; see Lemma~\ref{le.suporte}.

\subsection{Statement of results} The non-wandering set $\Omega_\phi$ of an impulsive semiflow $\phi$ may not be forward invariant, as the two examples in Section~\ref{se.examples} illustrate. This is the main difficulty that we have to deal with  to establish the existence of invariant probability measures for impulsive semiflows. To overcome this difficulty  we need to ensure that  points in $\Omega_\phi\setminus D$ which are close to $D$ necessarily have small first impulsive times. This corresponds to  the continuity of the function
%
%A way of ensuring this is imposing the continuity of a  slightly modified version  of the function $\tau_1$ restricted to $\Omega_\phi$. We introduce
%
$\tau_D:\Omega_\phi\to [0,+\infty]$, defined for $x\in\Omega_\phi$ as
$$
\tau_D(x)=
\begin{cases}
\tau_1(x), &\text{if $x\in \Omega_\phi\setminus D$};\\
0, &\text{if $x\in\Omega_\phi\cap D$.}
\end{cases}
$$
The continuity of $\tau_D$ also means that there is no $\varphi$-trajectory contained  in $\Omega_\phi$  passing through a point in~$D$.

To prove the existence of invariant probability measures for certain impulsive semiflows we must ensure in advance from properties of the map $I$ that the set $\Omega_\phi\setminus D$ is not empty. This amounts to request that, if the trajectory of a non-wandering point hits~$D$, then the instantaneous action of the impulse $I$ makes the orbit leave $D$ while staying in~$\Omega_\phi$.

\begin{maintheorem}\label{the.A}
Let $\phi$ be the semiflow of  an impulsive dynamical system $(X,\varphi, D, I)$. If~$\tau_D$ is continuous and $I(\Omega_\phi \cap D)\subset \Omega_\phi\setminus D$, then $\phi$ has some invariant probability measure. Moreover, any $\phi$-invariant probability measure $\mu$ has its support contained  in $\Omega_\phi$ and $\mu(D)=0$.
\end{maintheorem}

Notice that, under the condition $I(\Omega_\phi \cap D)\subset \Omega_\phi\setminus D$, we  necessarily have $\Omega_\phi\setminus D\neq\emptyset$.
Regarding the continuity of $\tau_D$ (or $\tau_1$), we consider the usual one-point compactification topology in $[0,+\infty)\cup\{+\infty\}$. As already mentioned in Remark~\ref{re.infty}, the function $\tau_1$ is always lower semicontinuous  on $X\setminus D$. Additionally, the upper semicontinuity of $\tau_1$ on $X\setminus D$ holds whenever the impulsive set $D$ satisfies some \emph{tube condition}; see \cite[Theorem~2.11]{C04a} for more details.

%As a  consequence of Lemma~\ref{le.hipoteses} below and Theorem~\ref{the.A} we obtain the following result.
%%Observe that the assumption $I(\Omega_\phi \cap D)\cap D=\emptyset$ holds in particular if $I(D)\cap D=\emptyset$.
%
%
%\begin{maincorollary}\label{co.B}
%Let $\phi$ be the semiflow of an impulsive dynamical system $(X,\varphi, D, I)$. If~$\tau_D$~is continuous and $I( D)\cap D=\emptyset$, then $\phi$ has some invariant probability measure.\end{maincorollary}

 As a byproduct of our strategy to prove Theorem~\ref{the.A} we obtain  the following result on the forward invariance of $\Omega_\phi\setminus D$.

\begin{maintheorem}\label{the.B}
Let $\phi$ be the semiflow of  an impulsive dynamical system $(X,\varphi, D, I)$.
 If  $I(\Omega_\phi \cap D)\subset \Omega_\phi\setminus D$, then %$\Omega_\phi\setminus D$ is forward invariant under $\phi$.
 $\phi_t(\Omega_\phi\setminus D)\subset \Omega_\phi\setminus D$ for all $t\ge 0$.
\end{maintheorem}

Observe that in the particular case when $\Omega_\phi\cap D=\emptyset$ we clearly have $I(\Omega_\phi \cap D)\subset \Omega_\phi\setminus D$, and so  it follows from Theorem~\ref{the.B} that $\phi$ induces a continuous semiflow on the compact set~$\Omega_\phi=\Omega_\phi\setminus D$. Hence, using Kryloff-Bogoliouboff Theorem \cite[Part II, Theorem I]{KB37}, we easily get the following result.

\begin{maincorollary}\label{co.novo}
Let $\phi$ be the semiflow of  an impulsive dynamical system $(X,\varphi, D, I)$.
  If  $\Omega_\phi\cap D = \emptyset$, then $\phi$ has some invariant probability measure.
\end{maincorollary}

Theorem~\ref{the.B} indicates that $\phi$ induces a semiflow on the set $\Omega\setminus D$. To prove Theorem~\ref{the.A} we shall show that the induced semiflow on  $\Omega\setminus D$ is conjugated by a bimeasurable invertible map to a continuous semiflow on a compact space. This is the content of the next result which, combined with Theorem~\ref{the.A}, asserts that, in measure theoretical terms, the dynamics on the most relevant part of the phase space of an impulsive  semiflow  can be seen as the dynamics of a continuous semiflow.

\begin{maintheorem}\label{the.extra}
Let $\phi$ be the semiflow of  an impulsive dynamical system $(X,\varphi, D, I)$ for which~$\tau_D$ is continuous and $I(\Omega_\phi \cap D)\subset \Omega_\phi\setminus D$. Then there are  a compact metric space~$Y$, a continuous semiflow $\psi$ in $Y$ and a continuous invertible bimeasurable map $h:\Omega_\phi\setminus D\to Y$ such that
$\psi_t\circ h\vert_{\Omega_\phi\setminus D} = h\circ \phi_t\vert_{\Omega_\phi\setminus D}$ for all $t\ge 0$.
Moreover, if $\iota:\Omega_\phi\setminus D \to X$ denotes the inclusion map, then $(\iota\circ h^{-1})_*:\mathcal M_\psi(Y)\to \mathcal M_\phi(X)$ is a bijection.
\end{maintheorem}

The first conclusion of Theorem~\ref{the.A} is now an easy consequence of Theorem~\ref{the.extra}. In fact, as~$\psi$ is continuous and $Y$ is a compact metric space, then $\psi$ has some invariant probability measure; see \cite[Part II, Theorem I]{KB37}. This means that $\mathcal M_\psi(Y)$ is nonempty, hence $\mathcal M_\phi(X)$ is nonempty as well.

The continuous bimeasurable map $h$ given by Theorem~\ref{the.extra}  allows us to exchange ergodic information between $\psi$ in $Y$ and $\phi$ in $\Omega_\phi$. For instance, any $\phi$-invariant probability measure gives measure zero to $D$ and the topological entropy of the continuous semiflow~$\psi$ is given by
$$h_{\top}(\psi) = h_{\top}(\psi_1)=\sup\,\{h_{\nu}(\psi_1): \nu \in M_\psi(Y)\},$$
 where $h_{\nu}(\psi_1)$ stands for the measure-theoretic entropy of the map $\psi_1$ with respect to the probability~$\nu$; see \cite{B75}. It thus follows that
$$ \sup\,\{h_{\mu}(\phi_1): \mu \in M_\phi(X)\}=h_{\top}(\psi).$$

%\end{remark}

%In particular, for any integrable function $f:X\to \R$ we have
% $$\int f d\mu= \int f\circ\pi^{-1}d\nu.$$
%To be formally rigorous, we should have written $\mu\vert_{\Omega_\phi\setminus D}$ instead of $\mu$ in the last equality of Theorem~\ref{the.extra}. However, from Theorem~\ref{the.A}, any $\phi$-invariant measure is supported on $\Omega_\phi$ and also $\mu(D)=0$, which means that there is no abuse  in thinking of  $\mu$ as $\mu\vert_{\Omega_\phi\setminus D}$.

\subsection{Overview} In Section~\ref{se.examples} we present two examples of impulsive semiflows, one with and another without an invariant probability measure. In Section~\ref{se.dynon} we reveal a few properties of a semiflow's dynamics on its non-wandering set. In particular, we prove Theorem~\ref{the.B}. In Section~\ref{se.quotient} we use the impulsive function to define an equivalence relation and thus obtain the space $Y$ of Theorem~\ref{the.extra} as the projection of the non-wandering set on the quotient space determined by that equivalence relation. We conclude the proofs of Theorem~\ref{the.A} and Theorem~\ref{the.extra} in Section~\ref{se.ipm}.

\section{Examples}\label{se.examples}

In this section  we present two examples of impulsive semiflows, the first one with an invariant probability measure and the second without such a measure. Our strategy in the first example is to guarantee that the  impulsive semiflow satisfies the assumptions of Theorem~\ref{the.A}.

\begin{example}[Impulsive system with an invariant probability measure]\label{ex.with}
Consider the phase space $X$ as the annulus
 $$X=\left\{(r\cos\theta,r\sin\theta)\in\R^2:  1\le r\le 2  , \,\theta \in [0,2\pi]\right\}
 $$
and define
  $$\varphi:\R^+_0\times X\to X$$
  as the semiflow of the vector field in $X$ given  by
  $$
   \begin{cases}
   r'=0 & \\
   \theta'=1 .&
   \end{cases}
$$
Note that the trajectories of $\varphi$ are circles spinning around zero counterclockwise. Then take
 $$D=\{(r,0)\in X: 1\le r \le 2\}
 $$
and define $I:D\to X$ by
 $$I(r,0)=\left(-\frac12-\frac12r, 0\right).
 $$
Let $\phi$ be the semiflow of the impulsive dynamical system $(X,\varphi, D, I)$. It is straightforward to verify that
 $$\Omega_\phi=\left\{(\cos\theta, \sin\theta): \frac{3\pi}{2}\le \theta\le 2\pi\right \} .$$
 Notice that $\Omega_\phi$ is not forward invariant under $\phi$, for the trajectory of $(1,0)\in\Omega_\phi$ is clearly not contained in $\Omega_\phi$. Still, we have $\Omega_\phi\setminus D\neq \emptyset$ and
  $$I(\Omega_\phi\cap D)=I(\{(1,0)\})=\{(-1,0)\}\subset \Omega_\phi\setminus D.$$
  Moreover, $\tau_D:\Omega_\phi\to [0,2\pi]$ is given for ${3\pi}/{2}\le \theta\le 2\pi$ by
     $$\tau_D(\cos\theta, \sin\theta)=2\pi - \theta,$$
   which is obviously continuous. Then, by Theorem~\ref{the.A}, the impulsive semiflow $\phi$ has some invariant probability measure.
  \end{example}

\begin{example}[Impulsive system with no invariant probability measure]
Consider again the phase space
 $$X=\left\{(r\cos\theta,r\sin\theta)\in\R^2:  1\le r\le 2  , \,\theta \in [0,2\pi]\right\},
 $$
but now define
  $$\varphi:\R^+_0\times X\to X$$
  as the semiflow associated to the vector field in $X$ given  by
  $$
   \begin{cases}
   r'=f(r) & \\
   \theta'=1 ,&
   \end{cases}
$$
where $f(r)= 1-r$, for $1\le r\le 2$. The trajectories of $\varphi$ are now curves spiraling counterclockwise and converging to the inner border circle of $X$.  Take  $$D=\{(1,0)\}$$
and $I:D\to X$ defined by
 $$I(1,0)=\left(2, 0\right).
 $$
If $\phi$ is the semiflow of the impulsive dynamical system $(X,\varphi, D, I)$, it is not difficult to see that
 $$\Omega_\phi=\left\{(\cos\theta, \sin\theta): 0\le \theta\le 2\pi\right \}$$
and that this set is not forward  invariant under $\phi$. In this case we have $\tau_D:\Omega_\phi\to [0,2\pi]$ given for $0< \theta \le 2\pi$ by
   $$\tau_D(\cos\theta, \sin\theta)=2\pi-\theta,$$
   which is clearly not continuous at the point $(1,0)$.

We claim that $\phi$ has no invariant probability measure. Actually, if $\mu$ were such a measure, then %$\mu$ would also be invariant by $\phi_{2\pi}$ and so we %, by Poincar\'e Recurrence Theorem \cite{H56},
by Lemma~\ref{le.suporte} we should have
$$1=\mu(\Omega_\phi) = \mu(\phi_{2\pi}^{-1}(\Omega_\phi)).$$
However, $\phi_{2\pi}^{-1}(\Omega_\phi)$ is the empty set.
\end{example}

\section{Dynamics on the non-wandering set}\label{se.dynon}

Here we present some  results on the non-wandering sets of semiflows.  The first result is well known for the so-called discrete time dynamical systems (or transformations). As we have not found a proof for semiflows in the literature, we  decided to include it here for the sake of completeness.

\begin{lemma}\label{le.suporte}
The support of a probability measure which is invariant by a semiflow is contained in its non-wandering set.
\end{lemma}

\begin{proof}
Let  $\Omega_\phi$ be the non-wandering set of  a semiflow $\phi$ on $X$ and $\mu$ a probability measure invariant by $\phi$. %Notice that, by definition $\mu$, is invariant by $\phi_s$, for all $s \geq 0$.
Assume that $x \notin \Omega_\phi$. Then there exists a neighborhood $V$ of $x$ in $X$ and $T > 0$ such that
 \begin{equation}\label{propV}
\phi_t^{-1}(V)\cap V=\emptyset,
 \quad\text{for all $t \geq T$.}
\end{equation}
 We claim that for all positive integers $m>n$, we have
$$\phi_{nT}^{-1}\,(V) \cap \phi_{mT}^{-1}\,(V)= \emptyset.$$
Otherwise, if $z \in \phi_{nT}^{-1}\,(V) \cap \phi_{mT}^{-1}\,(V)$, then
$$\phi_{nT}\,(z) \in V \quad \text{ and } \quad \phi_{(m-n)T}\,(\phi_{nT}\,(z))= \phi_{mT}\,(z) \in V,$$
contradicting \eqref{propV}.

We may now conclude that $\mu(V)=0$, thus confirming that $x$ is not in the support of~$\mu$. Indeed, if $\mu(V)$ were positive, then, using the invariance of $\mu$, we would have for any $n \in \mathbb{N}$
$$\mu(\phi_{nT}^{-1}\,(V))=\mu(V)$$
and
$$\mu\left(\bigcup_{n \,\in\, \mathbb{N}} \phi_{nT}^{-1}\,(V)\right) = \sum_{n \,\in\, \mathbb{N}}\mu(\phi_{nT}^{-1}\,(V)).$$
Hence
$$1\geq \mu\left(\bigcup_{n \,\in\, \mathbb{N}} \phi_{nT}^{-1}\,(V)\right)= \sum_{n \,\in\, \mathbb{N}}\mu(V)= +\infty,$$
which is an absurd.
\end{proof}

%\begin{lemma}\label{le.omegaD0}
%If  $\tau_D\vert_{\Omega_\phi}$ is continuous, then $I(\Omega_\phi \cap D)\subset \Omega_\phi$.
%\end{lemma}
%\begin{proof}
%Given $x\in \Omega_\phi\cap D$, consider a neighborhood $U$ of $I(x)$. As $I$ is continuous, there is some $\epsilon>0$ such that $I(B_\epsilon(x)\cap D)\subset U$, where $B_\epsilon(x)$ denotes the ball of radius $\epsilon$ centered at $x$. Take $t_0>0$ small such that for all $z\in X$
% \begin{equation}
%\label{eq.bola} \dist(z,\varphi_t(z))<\frac\epsilon2,\quad\forall 0\le t< t_0.
% \end{equation}
%By the continuity of $\tau_D\vert_{\Omega_\phi}$, we may choose $\delta<\epsilon/2$ such that
% $$\tau_D\vert_{B_\delta(x)}<t_0.$$
% As $x\in\Omega_\phi$ there is some $s>0$ and $y\in B_\delta(x)$ such that $z=\phi_s(y)\in B_\delta(x)$. In particular, we have $\tau_D(z)<t_0$, which means that there exists $t<t_0$ such that $\varphi_t(z)\in D$. It follows from \ref{eq.bola} and the choice of $\delta<\epsilon/2$ that
%  $$
%  \dist(\varphi_t(z),x)\le \dist(\varphi_t(z),z)+\dist(z,x)<\frac\epsilon2+\delta<\epsilon.
%  $$
%  Hence $I(\varphi_t(z))\in U$, and so
%\end{proof}
%
%Observe that in the previous result we used in an important way the continuity of $\tau_D\vert_{\Omega_\phi}$. For  the next result it suffices to have  the continuity of $$\tau_D\vert_{\Omega_\phi\setminus D}=\tau_1\vert_{\Omega_\phi\setminus D}.$$

Now we consider a result on non-wandering sets of impulsive semiflows which will be useful to prove Theorem~\ref{the.B}.

\begin{lemma}\label{le.omegaD0} Let $\phi$ be the semiflow of an impulsive dynamical system $(X,\varphi,D,I)$. Then
%Assume that  $\tau_D\vert_{\Omega_\phi\setminus D}$ is continuous. Then
 $\varphi_{t}(\Omega_\phi\setminus D)\subset \Omega_\phi$ for all  $0\le t\le \tau_1(x)$.
\end{lemma}

\begin{proof} Let $x\in \Omega_\phi\setminus D$. As $\varphi$ is continuous and $\Omega_\phi$ is closed, it is enough to show that
\begin{equation*}%\label{eq.enough}
\varphi_{t}(x)\in \Omega_\phi,\quad\text{for all $0< t< \tau_1(x)$}.
\end{equation*}
Take any $0< t< \tau_1(x)$ and consider an arbitrary neighborhood $U$  of $\varphi_{t}(x)$. We need to show that, given any $T>0$, there is $s>T$ such that
\begin{equation}\label{eq.enough}
\phi^{-1}_s(U)\cap U\neq\emptyset.
\end{equation}
By the continuity of $\varphi$ we have that  $V=\varphi_{t}^{-1}(U)$ is a neighborhood of $x$. Recalling that $\tau_1\vert_{X\setminus D}$ is lower semicontinuous by \cite[Theorem~2.7]{C04a}, then there exists a neighborhood $W$ of $x$ inside the open subset $X\setminus D$  such that, for each $y\in W$, we have
$\tau_1(y)>t.$ This implies that $\phi_s(y)=\varphi_s(y)$ for all $y\in W$ and all $s\le t.$ In particular,
$$\phi_t(V\cap W)=\varphi_t(V\cap W).$$
As $x\in\Omega_\phi$ and $V\cap W$ is a neighborhood of $x$, given $T>0$, there must be some  $s>T$ such that
 $$\phi_s^{-1}(V\cap W)\cap V\cap W\neq\emptyset .$$
 Therefore,
\begin{eqnarray*}
\emptyset&\neq &\phi_t(\phi_s^{-1}(V\cap W)\cap V\cap W )\\
&\subset &\phi_t(\phi_s^{-1}(V\cap W))\cap  \phi_t(V\cap W )\\
&\subset& \phi_s^{-1}(\phi_t(V\cap W))\cap  \phi_t(V\cap W )\\
&= &\phi_s^{-1}(\varphi_t(V\cap W))\cap  \varphi_t(V\cap W )\\
&\subset&\phi_s^{-1}(U)\cap U,
\end{eqnarray*}
and so we have proved \eqref{eq.enough}.
\end{proof}

Let us now prove  Theorem~\ref{the.B}.
Consider $\phi$ the semiflow of an impulsive dynamical system $(X,\varphi,D,I)$
for which   $I(\Omega_\phi \cap D)\subset \Omega_\phi\setminus D$. We need to verify that
\begin{equation}\label{eq.phinv}
\phi_t(\Omega_\phi\setminus D)\subset \Omega_\phi\setminus D,\quad\text{for all $t\ge 0$}.
\end{equation}
%
%
%\begin{lemma}\label{le.omegaD2} Let $\phi$ be the semiflow of an impulsive dynamical system $(X,\varphi,D,I)$.
%If  $I(\Omega_\phi \cap D)\subset \Omega_\phi\setminus D$, then $\phi_t(\Omega_\phi\setminus D)\subset \Omega_\phi\setminus D$ for all $t\ge 0$.
%\end{lemma}
%\begin{proof}
First of all, notice that we have $\phi_t(x)\notin D$ for all $x\in \Omega_\phi \setminus D$ and all $t>0$. Actually, by definition of the impulsive times, the only chance for   $\phi_t(x)$ to hit $D$ at a strictly positive time would be at an impulsive time. However, using Lemma~\ref{le.omegaD0} and the assumption $I(\Omega_\phi\cap D)\subset\Omega_\phi\setminus D$, we easily get that such an hitting on $D$ cannot happen.
Consequently, to prove our result it is enough to show that
\begin{equation}\label{menos}
\phi_t(\Omega_\phi\setminus D)\subset \Omega_\phi,\quad\forall\, t>0.
\end{equation}
Take any point  $x\in \Omega_\phi\setminus D$. If $\tau_1(x)=+\infty$, then, recalling that $\phi_t(x)=\varphi_t(x)$ for all $t>0$, by Lemma~\ref{le.omegaD0} we are done. Otherwise, it follows from Lemma~\ref{le.omegaD0} and the definition of  $\tau_1(x)$ that
$$ \phi_t(x)=\varphi_{t}(x)\in \Omega_\phi\setminus D,\quad\forall \,0< t<\tau_1(x),$$
and
$$\varphi_{\tau_1(x)}(x)\in \Omega_\phi\cap D.$$
Now, as we are assuming that $I(\Omega_\phi \cap D)\subset \Omega_\phi\setminus D$, it follows from the definition of $\phi$ that
 $$\phi_{\tau_1(x)}(x)=I(\varphi_{\tau_1(x)}(x))\in \Omega_\phi\setminus D.$$
Thus we have proved \eqref{menos} for $0<t\le \tau_1(x)$. We proceed inductively, repeating the same argument on the periods between the subsequent impulsive times of $x$, thus proving \eqref{eq.phinv}.

\section{Quotient dynamics}\label{se.quotient}
Given  an impulsive dynamical system $(X,\varphi,D,I)$, we
consider the quotient  space
 $ X/_\sim$,
where $\sim$ is the equivalence relation given by
 $$x\sim y\quad \Leftrightarrow \quad x=y, \quad y=I(x),\quad x=I(y)\quad \text{or}\quad I(x)=I(y).$$
We shall use $\tilde x$ to represent the equivalence class of $x\in X$. Consider $ X/_\sim$  endowed with the quotient topology and let
  $$\pi:X\to X/_\sim$$ be the natural projection.

   We are particularly interested in the set $\pi(\Omega_\phi)$, where $\phi$ is the impulsive semiflow of $(X,\varphi,D,I)$, for it will play an important role in the sequel.
As $\Omega_\phi$ is compact, then $\pi(\Omega_\phi)$ is a  pseudometric space; see $\S$23 of \cite{W70}. Actually, we shall verify in Lemma~\ref{le.compame} that~$\pi(\Omega_\phi)$ is %a $T_0$ space and so $\pi(\Omega_\phi)$ is in fact
a metric space.

Recall that a topological space $Y$ is said to be $T_0$ if, whenever $y_1$ and $y_2$ are distinct points in $Y$, there is an open set containing one and not the other.

\begin{lemma} A pseudometric $\rho$ in a space $Y$ is a metric if and only if $Y$, endowed with the topology $\rho$ generates, is $T_0$.
\end{lemma}

\begin{proof}
If the topology generated by $\rho$ makes $Y$ a $T_0$ space, then given $y_1\neq y_2$ in $Y$ there is an open subset, hence some $\epsilon$-ball, about one not containing the other. Therefore, $\rho(y_1,y_2)\geq \epsilon>0$, showing that $\rho$ is a metric.

Conversely, if $\rho$ is a metric, then any two distinct points are at a positive distance, say~$\epsilon>0$, and so the $\epsilon$-ball about one is an open set containing one and not the other.
\end{proof}

\begin{lemma}\label{le.compame}
$\pi(\Omega_\phi)$ is a compact metric space. %$T_0$.
\end{lemma}

\begin{proof}
The compactness of $\pi(\Omega_\phi)$ follows from the compactness of $\Omega_\phi$ and the continuity of the projection $\pi$. By the previous lemma, it is enough to show that $\pi(\Omega_\phi)$ with the quotient topology is a~$T_0$ space.
Given $\tilde x \in \pi(\Omega_\phi)$, the subset $C_x$ of $\Omega_\phi$ whose elements are in the equivalence class of~$x$ is a closed set in $\Omega_\phi$. Indeed,
$$
C_x=
\begin{cases}
\{x, I(x)\}\cup I^{-1}(\{x\})\cup I^{-1}(\{I(x)\}), &\text{if $x\in D$};\\
\{x\}\cup I^{-1}(\{x\}), &\text{otherwise.}
\end{cases}
$$
As $\Omega_\phi$ is a compact metric space, each one-point set is closed; moreover, as $I$ is continuous, both $I^{-1}(\{x\})$ and $I^{-1}(\{I(x)\})$ are closed. Thus $C_x$ is a finite union of closed sets, hence closed in $\Omega_\phi$. Therefore, the set $\{\tilde x\}$ is closed in $\pi(\Omega_\phi)$: its complement is open since
$$\pi^{-1}(\pi(\Omega_\phi) \setminus \{\tilde x\})= \Omega_\phi \setminus C_x$$
and $\Omega_\phi \setminus C_x$ is open in $\Omega_\phi$. Hence, given $\tilde y \neq \tilde x$ in $\pi(\Omega_\phi)$, the open set $\pi(\Omega_\phi) \setminus \{\tilde x\}$ contains~$\tilde y$ but not $\tilde x$.
\end{proof}

%   Also, if $d$ denotes the metric on $X$, the metric $\tilde d$ in $\pi(\Omega_\phi)$ that induces the quotient topology is given by
%$$ \tilde d \, (\tilde x,\tilde y) = \inf \, \{d\,(p_1, q_1)+ d\,(p_2, q_2)+\cdots +d\,(p_n,q_n)\},$$
%where $p_1, q_1, \dots, p_n, q_n$ is any chain of points in $\Omega_\phi$ such that
%$ p_1 \sim x$, $q_1 \sim p_2$, $q_2 \sim p_3$, ... $q_n \sim y$.
% In particular, we have $$\tilde d \, (\tilde x,\tilde y) \leq d\,(x,y), \quad\forall\text{ $x,y\in \Omega_\phi$.}$$

\begin{lemma}\label{le.contcomp}
Assume that  $\tau_D$ is continuous and $I(\Omega_\phi \cap D)\subset\Omega_\phi\setminus D$. Then
$\pi\circ\phi_t\vert_{\Omega_\phi\setminus D}$ is continuous for all $t\ge0$.
\end{lemma}

\begin{proof}
We start with the simple observation that either $\Omega_\phi \cap D = \emptyset$, in which case $\Omega_\phi\setminus D = \Omega_\phi \neq\emptyset$, or we have $I(\Omega_\phi \cap D)\subset\Omega_\phi\setminus D$ and so $\Omega_\phi\setminus D\neq\emptyset$.
As $I$ is continuous and $D$ is compact, then $I(\Omega_\phi\cap D)$ is compact. Now using the assumptions that $I(\Omega_\phi\cap D)\subset \Omega_\phi\setminus D$ and $\tau_1\vert_{\Omega_\phi \setminus D}=\tau_D\vert_{\Omega_\phi \setminus D}$ is strictly positive and continuous, there must be some constant $\alpha>0$ such that
\begin{equation}\label{eq.alpha}
\tau_1(z)>\alpha, \quad\text{for all } z\in I(\Omega_\phi \cap D).
\end{equation}

Given $t>0$,  let us prove the continuity of $\pi\circ\phi_t\vert_{\Omega_\phi\setminus D}$ at any point $x\in\Omega_\phi\setminus D$.  Using an inductive argument on the impulsive times of $x$, it is enough to show that, when $y\in\Omega_\phi\setminus D$ is close to~$x$, then $\pi(\phi_s(y))$ remains close to $\pi(\phi_s(x))$ for all $0\le s\le \tau_1(x)$. Notice that such an inductive argument on the impulsive times can be  applied because we are sure that $I(\Omega_\phi\cap D)\subset \Omega_\phi\setminus D$. The proof follows according to several cases:

\medskip

\noindent\textbf{Case 1}. $\tau_1(x)>t$.
\smallskip

\noindent As $\tau_D$ is continuous and $\tau_1$ coincides with $\tau_D$ in $\Omega_\phi\setminus D$, we must have  $\tau_1(y)>t$ for any point $y\in\Omega_\phi\setminus D$ sufficiently close to $x$. Therefore,  the result follows in this case from the continuity of the semiflow $\varphi$.

\medskip

\noindent\textbf{Case 2}.  $\tau_1(x)  \le  t$.
\smallskip

\noindent
 Given $y\in\Omega_\phi\setminus D$ sufficiently close to $x$, by the continuity of the semiflow~$\varphi$, the $\phi$-trajectories of $x$ and $y$ remain close until one of them hits the set~$D$. At this moment the impulsive function acts and, therefore, their $\phi$-trajectories may not remain close  at this first impulsive time. Now we distinguish three possible subcases:

\medskip

\noindent\textbf{Subcase 2.1}.  $\tau_1(x)=\tau_1(y)$.
\smallskip

\noindent
 The continuous map $I$ keeps the points $I(\varphi_{\tau_1(x)}(x)) $ and $ I(\varphi_{\tau_1(x)}(y))$ close, and this implies that  $\phi_s(x)$ and $\phi_s(y)$  remain close for all $0\le s\le \tau_1(x)$.

\medskip

\noindent\textbf{Subcase 2.2}.  $\tau_1(x)< \tau_1(y)$.
\smallskip

\noindent  By the continuity of $\varphi$ we have $\varphi_s(y)$ close to $\varphi_s(y)$ for~$y$ sufficiently close to~$x$ and $0\le s\le\tau_1(x)$.  This in particular implies that $\phi_s(y)$ is close to $\phi_s(y)$  for $0\le s<\tau_1(x)$.
It remains to check that $\pi(\phi_{\tau_1(x)}(y))$  is close to $\pi(\phi_{\tau_1(x)}(x))$. This is clearly true because
 $\varphi_{\tau_1(x)}(y)$ is close to  $\varphi_{\tau_1(x)}(x)$, and so
 $$\pi(\phi_{\tau_1(x)}(y))=\pi(\varphi_{\tau_1(x)}(y))$$ is close to
 $$\pi(\varphi_{\tau_1(x)}(x)) =\pi(I(\varphi_{\tau_1(x)}(x)))=\pi(\phi_{\tau_1(x)}(x)).$$

\medskip

\noindent\textbf{Subcase 2.3}.  $\tau_1(x)> \tau_1(y)$.
\smallskip

\noindent
By the continuity of $\varphi$ we have $\varphi_s(y)$ close to $\varphi_s(y)$ for~$y$ sufficiently close to~$x$ and $0\le s<\tau_1(y)$. This in particular implies that $\phi_s(y)$ is close to $\phi_s(y)$  for $0\le s<\tau_1(y)$. It remains to check that $\pi(\phi_{s}(y))$  is close to $\pi(\phi_{s}(x))$ for $\tau_1(y)\le s\le \tau_1(x)$.
By Lemma~\ref{le.omegaD0} and the definition of  first impulsive time we have
$$\varphi_{\tau_1(y)}(y) \in \Omega_\phi \cap D .$$%\quad\text{and}\quad \varphi_{\tau_1(y)}(x)  \in \Omega_\phi \setminus D.$$
%Observe that $x_1$ is accumulated by $\phi_t(p_1)$ for $0 < t < \tau_1(p_1)$, in this interval the points $\phi_t(p_1)$ are in $\Omega_\phi$ due to Lemma~\ref{le.omegaD2}, and $\Omega_\phi$ is closed.
 As we are assuming  $I(\Omega_\phi \cap D)\subset \Omega_\phi\setminus D$, we have in particular
$\phi_{\tau_1(y)}(y)=I(\varphi_{\tau_1(y)}(y)) \in \Omega_\phi,$
which by \eqref{eq.alpha} gives
 $$\tau_1(\phi_{\tau_1(y)}(y))>\alpha.$$
Using that $\tau_D$ is continuous at~$x$, we have
$\tau_D(x)-\tau_D(y) $ small for $y$ close to $x$.
We may have in particular
$$\tau_D(x)-\tau_D(y) <\alpha$$
Hence, for $\tau_1(y)\le s\le\tau_1(x)$ we have
$$\phi_s(y)=\varphi_{s-\tau_1(y)}(\phi_{\tau_1(y)}(y))=\varphi_{s-\tau_1(y)}(I(\varphi_{\tau_1(y)}(y)))$$
Observing that $s-\tau_1(y)\le \tau_1(x)-\tau_1(y)$ is close to $0$ for $y$ close to $x$, we have
$$\varphi_{s-\tau_1(y)}(I(\varphi_{\tau_1(y)}(y)))\quad\text{ close to }\quad I(\varphi_{\tau_1(y)}(y)).$$
Hence for $\tau_1(y)\le s\le\tau_1(x)$ we have
$$\pi(\phi_s(y))\quad\text{ close to }\quad\pi(I(\varphi_{\tau_1(y)}(y)))=\pi(\varphi_{\tau_1(y)}(y)).$$
Now we just need to notice that,  for $\tau_1(y)\le s\le\tau_1(x)$,  we have
$\varphi_{\tau_1(y)}(y)$ close to
$\varphi_{s}(y)$ which is itself close to $\varphi_{s}(x)$.
This way, we get, for $\tau_1(y)\le s\le\tau_1(x)$,
 $$\pi(\varphi_{\tau_1(y)}(y)) \quad\text{ close to }\quad  \pi(\varphi_{s}(x))= \pi(\phi_{s}(x)).$$
Lastly, recall that for $s=\tau_1(x)$ we have $  \pi(\varphi_{\tau_1(x)}(x))=\pi(I(\varphi_{\tau_1(x)}(x)))= \pi(\phi_{\tau_1(x)}(x)). $
\end{proof}

The first part of Theorem~\ref{the.extra} follows from the next proposition taking $Y=\pi(\Omega_\phi)$.

\begin{proposition}\label{pr.omega}
Assume that  $\tau_D$ is continuous and $I(\Omega_\phi \cap D)\subset\Omega_\phi\setminus D$. Then $\pi\vert_{\Omega_\phi\setminus D}$ is a continuous bimeasurable bijection onto $\pi(\Omega_\phi)$ and there exists a continuous semiflow $\psi:\R^+_0\times \pi(\Omega_\phi)\to\pi(\Omega_\phi)$ such that
   for all $ t\ge 0$
     \begin{equation}\label{eq.conjuga}
\psi_t\circ \pi\vert_{\Omega_\phi\setminus D} = \pi\circ \phi_t\vert_{\Omega_\phi\setminus D}.
\end{equation}
\end{proposition}

\begin{proof}
Assuming that $I(\Omega_\phi \cap D)\subset \Omega_\phi\setminus D$,  from the definition of the equivalence relation one easily deduces that
$$\pi(\Omega_\phi\setminus D)=\pi(\Omega_\phi).$$
Additionally,  for any $x,y\in \Omega_\phi\setminus D$ we have $x\sim y$ if and only if $x=y$.
This shows that $\pi\vert_{\Omega_\phi\setminus D}$ is a continuous bijection (not necessarily a homeomorphism) from $\Omega_\phi\setminus D$ onto $\pi(\Omega_\phi)$.
On the other hand, as $\pi\vert_{\Omega_\phi \setminus D}$ is injective, the image under the continuous map  $\pi$ of any Borel set is a Borel set; see \cite{P66}.
Moreover, from Theorem~\ref{the.B} we have  $\phi_t(\Omega_\phi\setminus D)\subset \Omega_\phi\setminus D$ for all $t\ge 0$. Then, setting
 $$\psi(t,\tilde x)=\pi(\phi(t,x))$$
 for
each $x\in\Omega_\phi\setminus D$ and $t\ge 0$,
we have that $\psi:\R^+\times \pi(\Omega_\phi )\to \pi(\Omega_\phi)$ is well defined and
obviously satisfies for all $t\ge0$
\begin{equation}\label{eq.conjuga2}
\psi_t\circ \pi\vert_{\Omega_\phi\setminus D} = \pi\circ \phi_t\vert_{\Omega_\phi\setminus D}.
\end{equation}
We are left to prove that $\psi$ is continuous. Considering for each $\tilde x\in \pi(\Omega_\phi )$ the map $\psi^{\tilde x}: \R^+_0\to \pi(\Omega_\phi )$ defined by
 $$\psi^{\tilde x}(t)=\psi(t,\tilde x),$$
it  is enough to prove that $\psi^{\tilde x}$ and $\psi_t$ are continuous for all $\tilde x\in \pi(\Omega_\phi )$ and all $t\ge0$.

Let  us start by proving the continuity of $\psi^{\tilde x}$ for   $x\in \Omega_\phi\setminus D$.
Consider first $t_0\ge0$  which is not an impulsive time for $x$. In this case we have, for $t$ in a sufficiently small neighborhood of $t_0$ in $\R^+_0$,
 $$\psi^{\tilde x}(t)=\pi(\varphi(t,x))$$
and, as $\varphi$ is continuous, this obviously gives the continuity of $\psi^{\tilde x}$ at $t_0$. On the other hand, if $t_0$ is an impulsive time for $x$, then we have
 $$\lim_{t\to t_0^-}\psi^{\tilde x}(t)=\lim_{t\to t_0^-}\pi(\phi(t,x))=\lim_{t\to t_0^-}\pi(\varphi(t,x))
= \pi(\varphi(t_0,x)).$$
 As $\varphi(t_0,x)\in D$,  it follows from the definition of $\phi(t_0,x)$  and the equivalence relation that yields the projection $\pi$ that
   $$ \pi(\varphi(t_0,x))=\pi(I(\varphi(t_0,x)))=\pi(\phi(t_0,x))=\psi^{\tilde x}(t_0).$$
This gives the continuity of $\psi^{\tilde x}$ on the left hand side of $t_0$. The continuity on the right hand side of $t_0$ follows easily from the fact that, by definition, the impulsive trajectories  are continuous on the right hand side.

Let us now prove the continuity of $\psi_t$ for $t\ge0$.
Notice that as we are considering the quotient topology in $\pi(\Omega_\phi)=\pi(\Omega_\phi\setminus D)$, we know that $\psi_t$ is
continuous if and only if $\psi_t\circ\pi\vert_{\Omega_\phi\setminus D}$ is continuous.
The continuity of $\psi_t\circ\pi\vert_{\Omega_\phi\setminus D}$ is an immediate consequence of Lemma~\ref{le.contcomp} and \eqref{eq.conjuga2}.
\end{proof}

\section{Invariant probability measures}\label{se.ipm}

In this section we prove the second parts of both Theorem~\ref{the.A} and Theorem~\ref{the.extra}. We  start with a general result on the measurability of impulsive semiflows.

\begin{proposition}
If $\phi$ is the semiflow of  an impulsive dynamical system $(X,\varphi, D, I)$, then $\phi$ is measurable.
\end{proposition}
\begin{proof}
We start by observing  that, as $\varphi$ is continuous and  $D$ is a Borel set, then the  impulsive time functions $\tau_1,\tau_2,\dots$ are all measurable.
%For the sake of notational simplicity, it will be convenient to consider $\tau_0(x)=0$ for all $x\in X$.
We define
 $$S_0=\left\{ (t,x)\in \R^+_0\times X: 0\le t<\tau_{1}(x)\right\} $$
and, for each $n\ge 1$,
\begin{eqnarray*}
S_n &=& \left\{ (t,x)\in \R^+_0\times X: \tau_{n}(x)<t<\tau_{n+1}(x)\right\} \\
T_n &=& \left\{ (t,x)\in \R^+_0\times X: \tau_n(x)=t\right\}.
\end{eqnarray*}
Notice that,  by the measurability of the impulsive time functions, these are Borel sets.

Now, given any Borel set $A\subset X$, we may write $\phi^{-1}(A)$ as a disjoint union of sets of the types
 $$\phi^{-1}(A)\cap S_0, \quad \phi^{-1}(A)\cap S_n\quad\text{and}\quad \phi^{-1}(A)\cap T_n, \quad n\ge 1.$$
 Moreover,
 $$ \phi^{-1}(A)\cap S_0 = \left\{(t,x)\in \R^+_0\times X: \varphi_{t}(x)\in A \right\} \cap S_0$$
  and, for each $n\ge1$,
\begin{eqnarray*}
 \phi^{-1}(A)\cap S_n & = &  \left\{ (t,x)\in S_n:  \varphi_{t-\tau_{n}(x)}\circ I \circ \varphi_{\tau_n(x)-\tau_{n-1}(x)} \circ \cdots \circ I\circ\varphi_{\tau_1(x)}(x)\in A \right\} \\
 \phi^{-1}(A)\cap T_n & = &  \left\{ (t,x)\in T_n: I \circ \varphi_{\tau_n(x)-\tau_{n-1}(x)} \circ  \cdots \circ I\circ\varphi_{\tau_1(x)}(x)\in A \right\}.
\end{eqnarray*}
Taking into account the measurability of the functions $\varphi$, $I$ and $\tau_n$, for $n\ge 1$, all the sets considered above are Borel sets. This ensures the measurability of $\phi$.
\end{proof}

In the next result we show that, under the assumptions of Theorem~\ref{the.A}, we have $\mu(D)=0$ for any $\phi$-invariant measure $\mu$, thus obtaining the second conclusion of Theorem~\ref{the.A}. Notice that, assuming $I(D)\cap D=\emptyset$, we can prove this assertion quite easily. Indeed, supposing by contradiction  that $\mu(D)>0$,  it follows from Poincar\'e Recurrence Theorem that for~$\mu$ almost every $x\in D$ there are infinitely many moments $t>0$ such that
 $\phi_t(x)\in  D.
$
Clearly, if  $I(D)\cap D=\emptyset$, then the $\phi$-trajectories do not hit $D$ for $t>0$, and so we reach a contradiction.

\begin{lemma}\label{le.mud}
 Let $\mu$ be an invariant probability measure for the semiflow $\phi$.
If  $\tau_D$ is continuous and $I(\Omega_\phi \cap D)\subset \Omega_\phi\setminus D$, then $\mu(D)=0$.
\end{lemma}
\begin{proof}
Assume, by contradiction,  that $\mu(D)>0$. Since the support of $\mu$ is contained in $\Omega_\phi$, by Lemma~\ref{le.suporte}, one necessarily has $\mu(X\setminus \Omega_\phi)=0$. This implies that
 \begin{equation}\label{eq.omegD}
\mu(\Omega_\phi\cap D)>0.
\end{equation}
Defining, for each $n\ge 1$,
 $$D_n=\left\{x\in \Omega_\phi\cap D: \tau_1(x)>\frac1n\right\},$$
and observing that $\tau_1$ is strictly positive, by assumption, we clearly have $$\Omega_\phi\cap D=\bigcup_{n\ge 1} D_n.$$
 We claim that, for each $n\ge 1$, there must be some $0<\epsilon_n<1/n$ such that
  \begin{equation}\label{eqphiomega}
\phi_{t}(D_n)\cap \Omega_\phi=\emptyset,\quad\text{for all $0< t\le\epsilon_n$}.
\end{equation}
Actually, assuming by contradiction that \eqref{eqphiomega} does not hold,  there are  sequences  $x_k\in D_n$ and $t_k\to 0$, with $0<t_k<1/n$,  such that $\phi_{t_k}(x_k)\in\Omega_\phi$ for all $k\ge1$. Using that $\Omega_\phi\cap D$ is compact and taking a converging subsequence, if necessary, we may assume that $\phi_{t_k}(x_k)$ converges to some point $x\in \Omega_\phi$ when $k\to\infty$. Then, the continuity of $\tau_D$ yields
 \begin{equation}\label{eq.conver}
\tau_D( \phi_{t_k}(x_k) )\longrightarrow\tau_D(x)=0,\quad\text{as $k\to\infty$}.
\end{equation}
Now, observing that for $x_k\in D_n$ we must  have $\phi_{t}(x_k) \notin D$ for all $0<t<1/n$, it follows that
 $$
 \tau_D( \phi_{t_k}(x_k) )  = \tau_1( \phi_{t_k}(x_k) )>\frac1n-t_k.
  $$
  Recalling that $t_k\to0$, this last formula gives a contradiction with \eqref{eq.conver}, and so \eqref{eqphiomega} must hold.

  Now, from the regularity of measures defined on compact metric spaces, for each $n\ge 1$ there is a compact set $K_n\subset D_n$ such that
  $$\mu(D_n)\le \mu(K_n)+ \frac{\mu(\Omega_\phi\cap D)}{2^{n+1}}.$$
  Hence
   \begin{equation}\label{eq.ok}
\mu(\Omega_\phi\cap D)  \le \sum_{n\ge 1}\mu(D_n)\le
\sum_{n\ge 1}\mu(K_n) +  \frac{\mu(\Omega_\phi\cap D)}{2}.
\end{equation}
Observing that as $K_n$ is compact and $\phi_{\epsilon_n}$ is continuous, then $\phi_{\epsilon_n}(K_n)$ is compact  and thus a Borel set, we may write
$$
\sum_{n\ge 1}\mu(K_n)\le \sum_{n\ge 1}\mu(\phi_{\epsilon_n}^{-1}(\phi_{\epsilon_n}(K_n)) )=  \sum_{n\ge 1}\mu(\phi_{\epsilon_n}(K_n)),
 $$
 where in this last equality we have used the $\phi$-invariance of $\mu$.
Taking also into account that \eqref{eqphiomega} necessarily implies that $\phi_{\epsilon_n}(K_n)\subset X\setminus  \Omega_\phi$, it follows that $\mu(\phi_{\epsilon_n}(K_n))=0$  for all $n\ge1$, by Lemma~\ref{le.suporte}. From \eqref{eq.ok} we get
 $$
 \mu(\Omega_\phi\cap D)  \le   \frac{\mu(\Omega_\phi\cap D)}{2},
$$
 which gives a contradiction with \eqref{eq.omegD}.
  \end{proof}

We are now ready to conclude the proof of Theorem~\ref{the.extra}.  Let  $\phi$ be the semiflow of an impulsive dynamical system $(X,\varphi, D, I)$  for which
$\tau_D$ is continuous and    $I(\Omega_\phi \cap D)\subset \Omega_\phi\setminus D$.
Consider
 $$
\begin{array}{ccc}
       h\, :\, \Omega_\phi \setminus D& \longrightarrow &\pi(\Omega_\phi) \\
         x & \longmapsto & \pi(x)
        \end{array}$$
the continuous  bimeasurable bijection and
$\psi$ the continuous semiflow in $\pi(\Omega_\phi)$  given by Proposition~\ref{pr.omega}. Consider also  $\iota:\Omega_\phi\setminus D \to X$ the inclusion map.

 Before we state our final lemma, let us list some useful properties about push-forwards which are a straightforward  consequence of the  definitions given in Subsection~\ref{se.im}:

\begin{enumerate}
\item[(P1)] If $f:X_1\to X_2$ and $g:X_2\to X_3$ are measurable, then
\begin{equation*}\label{eq.fg}
(g\circ f)_*=g_*\circ f_*.
\end{equation*}
\item[(P2)] If $f:X_1\to X_2$ is invertible and bimeasurable, then $f_*$ is invertible and
 \begin{equation*}\label{eq.invertible}
f_*^{-1}=(f^{-1})_*.
\end{equation*}
\item[(P3)]  If  $f:X_1\to X_2$  is measurable and $\phi^1, \phi^2$ are semiflows in $X_1, X_2$, respectively,  such that for all $t\ge0$
 $$\phi^2_t\circ f=f\circ \phi^1_t,$$
 then
\begin{equation*}%\label{eq.conjpush}
\mu\in\mathcal M_{\phi^1}(X_1)\quad\Rightarrow\quad f_* \mu\in\mathcal M_{\phi^2}(X_2).
\end{equation*}
\end{enumerate}

%Before we state our next result, let us see some useful properties about push-forwards. It easily follows from the definition that if $f:X_1\to X_2$ and $g:X_2\to X_3$ are measurable maps, then
%\begin{equation}\label{eq.fg}
%(g\circ f)_*=f_*\circ g_*.
%\end{equation}
%Moreover, supposing that $f$ is bimeasurable, then $f_*$ is invertible and
% \begin{equation}\label{eq.invertible}
%f_*^{-1}=(f^{-1})_*.
%\end{equation}
% Finally, if $\phi^1$ and $\phi^2$ are semiflows in $X_1$ and $X_2$, respectively, such that
% $$\phi^1_t\circ f=f\circ \phi^2_t,\quad \text{for all $t\ge0$},$$
% then
%\begin{equation}\label{eq.conjpush}
%\mu\in\mathcal M_{\phi^2}(X_2)\quad\Rightarrow\quad f_* \mu\in\mathcal M_{\phi^1}(X_1).
%\end{equation}

The second part of Theorem~\ref{the.extra} follows immediately from Lemma~\ref{le.compame} and the next result with $Y=\pi(\Omega_\phi)$.

 \begin{lemma}\label{eq.conjpush}
$ (\iota\circ h^{-1})_*: \mathcal M_\psi(\pi(\Omega_\phi))  \longrightarrow \mathcal M_\phi(X)$ is well defined and is a bijection.
 \end{lemma}

 \begin{proof}
 To see that $ (\iota\circ h^{-1})_*$ is well defined, we need to check that if $\nu\in \mathcal M_\psi(\pi(\Omega_\phi))$, then one necessarily has $ (\iota\circ h^{-1})_*\nu\in M_\phi(X)$.
Recalling  Theorem~\ref{the.B}, we may define  the ``restricted" semiflow $\hat\phi:\R^+_0\times (\Omega_\phi\setminus D)\to \Omega_\phi\setminus D$ as
  $$\hat\phi_t(x)=\phi_t(x).$$
We clearly have
  \begin{equation}\label{eq.primaconj}
\phi_t\circ\iota =\iota\circ\hat\phi_t ,\quad\text{for all $t\ge0$}.
\end{equation}
Equation \eqref{eq.conjuga} can now be restated as
 $$h\circ\hat\phi_t=\psi_t\circ h,\quad\text{for all $t\ge0$},$$
 or equivalently
 \begin{equation}\label{eq.menosum}
 \hat\phi_t\circ h^{-1}=h^{-1}\circ\psi_t,\quad\text{for all $t\ge0$}.
\end{equation}
Now it follows from  (P3), \eqref{eq.primaconj} and \eqref{eq.menosum}  that
 $$\nu\in \mathcal M_{\psi}(\pi(\Omega_\phi)) \quad\Rightarrow\quad (h^{-1})_*\nu \in \mathcal M_{\hat\phi}(\Omega_\phi\setminus D) \quad\Rightarrow\quad \iota_*(h^{-1})_*\nu\in \mathcal M_{\phi}(X). $$
Finally, using (P1) we obtain $\iota_*(h^{-1})_*=(\iota\circ h^{-1})_*$, and so $(\iota\circ h^{-1})_*\nu\in \mathcal M_{\phi}(X)$. Hence $ (\iota\circ h^{-1})_*$ is well defined.

We are left to check that $(\iota\circ h^{-1})_*$ is a bijection. As $(\iota\circ h^{-1})_*=\iota_*(h^{-1})_*$ and
$(h^{-1})_*$ is invertible by (P2), we only have to show that $\iota_*$ is invertible. Clearly, $\iota$ being injective it has a left inverse, which implies, by (P1),  that $\iota_*$ has a right inverse, and so $\iota_*$ is surjective. We are left to prove  that $\iota_*$ is injective. Let  $\mu_1,\mu_2\in \mathcal M_{\hat\phi}(\Omega_\phi\setminus D)$ be such that $\iota_*\mu_1=\iota_*\mu_2$. This implies that
  for  any Borel set $A\subset X$ we have
  \begin{equation}\label{eq.mu12}
\mu_1(A\cap (\Omega_\phi\setminus D) )=\mu_2(A\cap (\Omega_\phi\setminus D) ).
\end{equation}
Finally, using  Lemma~\ref{le.suporte} and Lemma~\ref{le.mud} it easily follows from \eqref{eq.mu12} that $\mu_1(A)=\mu_2(A)$, and so  $\iota_*$ is injective.
\end{proof}

\end{document}